%% file: lls-connected.tex
\documentclass{amsart}
\include{headers}

\begin{document}
\title{Connectedness of Brill-Noether loci via degenerations}
\author{Brian Osserman}
\begin{abstract} We show that limit linear series spaces for chains of
curves are reduced. Using new advances in the foundations of limit linear
series, we then use degenerations to study the question of connectedness for 
spaces of linear series with imposed ramification at up to two points.
We find that in general, these spaces may not be connected even when they
have positive dimension, but we prove a criterion for connectedness which 
generalizes the theorem previously proved by Fulton and Lazarsfeld in the 
case without imposed ramification.
\end{abstract}

\thanks{The author is partially supported by a grant from the Simons 
Foundation \#279151.}

\maketitle

\section{Introduction}

The classical Brill-Noether theorem states that if we are given 
$g,r,d \geq 0$, a general curve $X$ of genus $g$ carries a linear series
$(\sL,V)$ of rank $r$ and degree $d$ if and only if the quantity
$$\rho:=g-(r+1)(r+g-d)$$
is nonnegative \cite{g-h1}. Moreover, in this case 
the moduli space $G^r_d(X)$ of such linear series has pure dimension
$\rho$. This was generalized by Eisenbud and Harris
to allow for imposed ramification: if we have marked points 
$P_1,\dots,P_n \in X$, and sequences $0<a^i_0<\dots<a^i_r \leq d$ for
$i=1,\dots,n$, we can consider the moduli space 
$G^r_d(X,(P_1,a^1_{\bullet}),\dots,(P_n,a^n_{\bullet})) \subseteq
G^r_d(X)$ parametrizing linear series with vanishing sequence at least
$a^i_{\bullet}$ at each of the $P_i$. Then Eisenbud and Harris used their
theory of limit linear series to show
that in characteristic $0$, if $(X,P_1,\dots,P_n)$ is a general $n$-marked 
curve of genus $g$, the dimension of
$G^r_d(X,(P_1,a^1_{\bullet}),\dots,(P_n,a^n_{\bullet}))$ -- if it is nonempty
-- is given by the generalized formula
$$\rho:=g-(r+1)(r+g-d)-\sum_{i=1}^n \sum_{j=0}^r (a^i_j-j).$$
The condition for nonemptiness is still numerical, but becomes more 
complicated in this context. This theorem fails in positive characteristic
for $n \geq 3$, but is still true if $n \leq 2$; in this case, we also
have a simple criterion for nonemptiness. See for instance
\cite{os18} for a proof of the following.

\begin{thm}\label{thm:bg} Given $(g,r,d)$ nonnegative integers,
and sequences $0 \leq a_0<a_1 <\dots<a_r \leq d$,
$0 \leq b_0<b_1 <\dots<b_r \leq d$, 
set
\begin{equation}\label{eq:hat-rho}
\widehat{\rho}:=g-\sum_{j:a_j+b_{r-j} > d-g} a_j+b_{r-j}-(d-g).
\end{equation}
Then, if $X$ is a general (smooth, projective)
curve of genus $g$ and $P,Q \in X$ are general points, the moduli
space $G^r_d(X,(P,a_{\bullet}),(Q,b_{\bullet}))$
is nonempty if and only if $\widehat{\rho} \geq 0$,
and in this case, it has pure dimension $\rho$.
\end{thm}

In a complementary direction, Fulton and Lazarsfeld used an analysis
of degeneracy loci on $\Pic^d(X)$ to show that the spaces $G^r_d(X)$ 
are always connected when $\rho \geq 1$ \cite{f-l1}.
Our main result is the following theorem combining these two strands:

\begin{thm}\label{thm:main} In the situation of Theorem \ref{thm:bg},
the space $G^r_d(X,(P,a_{\bullet}),(Q,b_{\bullet}))$
is also reduced whenever it is nonempty.

Furthermore, if $(X,P,Q)$ is any $2$-marked curve of genus $g$ such that
the space $G^r_d(X,(P,a_{\bullet}),(Q,b_{\bullet}))$ is pure of dimension 
$\rho$,
and if $\widehat{\rho} \geq 1$, then $G^r_d(X,(P,a_{\bullet}),(Q,b_{\bullet}))$ 
is connected.
\end{thm}

Note that in the situation of the theorems, $\rho$ is given by the same
formula as \eqref{eq:hat-rho}, but with the sum ranging over all $j$.
Thus, $\widehat{\rho} \leq \rho$ always. We will have
$\widehat{\rho}=\rho$ whenever $d \leq r+g$, which underlines that for questions
involving imposed ramification, the nonspecial case often displays 
interesting and novel behavior. On the other hand, when no ramification
is imposed, we will have $\widehat{\rho} =g$ whenever $d>r+g$, so we recover
the Fulton-Lazarsfeld result on connectedness whenever $\rho \geq 1$.
Example \ref{ex:disconnected} demonstrates that in the presence of imposed 
ramification, the hypothesis that $\widehat{\rho} \geq 1$ is indeed necessary 
for connectedness. Together with the criterion for nonemptiness stated in
Theorem \ref{thm:bg}, a pattern emerges that in generalizing from
classical statements without imposed ramification, the quantity 
$\widehat{\rho}$ seems to arise naturally as an alternative generalization
of $\rho$.

Aside from the generalization to imposed ramification, the novelty of
our approach is that until now, the connectedness theorem was the only
part of standard Brill-Noether theory which did not have a proof using
degeneration techniques. The main reason for this was not that the 
analysis of the relevant limit linear series spaces was especially 
difficult, but that to study topological properties such as connectedness
via specializations, it is crucial to have flat, proper families with
well-understood scheme structures. Until the machinery introduced in 
\cite{os20} and \cite{o-m1}, such moduli spaces had never been constructed 
in the context of limit linear series. It is also a pleasant facet of our 
techniques that they are intrinsically characteristic independent.
Finally, we mention that even when $\widehat{\rho}=0$, our techniques reduce
describing the number of connected components of 
$G^r_d(X,(P,a_{\bullet}),(Q,b_{\bullet}))$ to a purely combinatorial
problem; see Remark \ref{rem:conn-combinatorial}.

Note that our connectedness theorem does not strictly generalize the 
Fulton-Lazarsfeld theorem because it applies to Brill-Noether-general 
curves, whereas their result applies to all curves. While reducedness was
known in the absence of imposed ramification going back to Griffiths
and Harris \cite{g-h1}, and is in any case superceded by the Gieseker-Petri
theorem, it appears to be new in the case of imposed
ramification, except that the case $\rho \leq 1$ is proved by Chan,
Lopez, Pflueger, and Teixidor i Bigas in \cite{c-l-p-t1}. The case of
ramification at three or more points is much more delicate; see Remark
\ref{rem:more-ram}.

In the case that $\rho=0$ and there is no imposed ramification, spaces
of linear series are not connected, but (in characteristic $0$) Eisenbud 
and Harris proved \cite{e-h7} that in suitable families, the relative
space of linear series is still connected. This leads us to ask:

\begin{ques} In cases with $\widehat{\rho}=0$, when the space of linear
series with imposed ramification on individual curves is not connected,
are there families of curves over which the relative space of linear series
is nonetheless connected?
\end{ques}

\subsection*{Acknowledgements} I would like to thank Izzet Coskun and 
Frank Sottile for helpful conversations, and Melody Chan and Nathan 
Pflueger for their comments, especially pointing out an important oversight 
in an earlier version of this paper.

\subsection*{Conventions}

We work throughout over an algebraically closed field of arbitrary
characteristic.

Given $r,d\geq 0$, a \textbf{vanishing sequence} $a_{\bullet}=a_0,\dots,a_r$
is a strictly increasing sequence with $a_0 \geq 0$ and $a_r \leq d$.
Two vanishing sequences $a_{\bullet}$, $b_{\bullet}$ are 
\textbf{complementary} if $a_j+b_{r-j}=d$ for $j=0,\dots,r$.
Given vanishing sequences $a_{\bullet},a'_{\bullet}$, we write
$a_{\bullet} \geq a'_{\bullet}$ if $a_j \geq a'_j$ for $j=0,\dots,r$, and
we write $a_{\bullet}>a'_{\bullet}$ if $a_{\bullet} \geq a'_{\bullet}$ and
furthermore $a_j>a'_j$ for some $j$.

\section{Richardson varieties and the base cases}

We recall that a \textbf{Richardson variety} is an intersection of two
Schubert varieties associated to transverse flags. These are well studied,
starting with Richardson \cite{ri2}: they are irreducible and
reduced of the expected codimension; it follows from
the Cohen-Macaulayness of Schubert varieties that
they are also Cohen-Macaulay. Richardson varieties arise naturally in 
studying the fibers of the map
$$G^r_d(X,(P,a_{\bullet}),(Q,b_{\bullet})) \to \Pic^d(X);$$ 
in fact, in our base cases where $X$ has genus $0$ or $1$, the general
fibers will always be Richardson varieties. This is the basis for the proof 
of the following.

\begin{prop}\label{prop:base-cases} Given $r,d$ nonnegative integers,
and sequences $0 \leq a_0<a_1 <\dots<a_r \leq d$,
$0 \leq b_0<b_1 <\dots<b_r \leq d$, let $X$ be a smooth, projective
curve of genus $g$ and $P,Q \in X$, and suppose either that $g=0$ and
$P \neq Q$, or that $g=1$ and $P-Q$ is not $m$-torsion for any $m \leq d$.
Then if $\widehat{\rho}\geq 0$,
the moduli space $G^r_d(X,(P,a_{\bullet}),(Q,b_{\bullet}))$ is 
reduced and irreducible, and Cohen-Macaulay of dimension $\rho$.
\end{prop}

It is already shown for instance in Lemma 2.1 of \cite{os18} that under
the hypotheses of Proposition \ref{prop:base-cases}, the space
$G^r_d(X,(P,a_{\bullet}),(Q,b_{\bullet}))$ is nonempty if and only if
$\widehat{\rho}$ is satisfied, and in this case it has pure dimension
$\rho$. As indicated above, the basic idea is to describe most of the fibers 
of the morphism
$G^r_d(X,(P,a_{\bullet}),(Q,b_{\bullet})) \to \Pic^d(X)$ as Richardson
varieties, and a more careful analysis of this argument will yield a proof
of Proposition \ref{prop:base-cases}. In fact, there was an oversight in
one case of the argument in \cite{os18}, which we will address 
simultaneously.

\begin{proof}
First, if $X$ has genus $0$, the space
$G^r_d(X,(P,a_{\bullet}),(Q,b_{\bullet}))$ is identified
with an intersection of two Schubert varieties in the Grassmannian
$G(r+1,\Gamma(X,\sO(d)) \cong G(r+1,d+1)$. Moreover, one checks easily that 
the flags defining
these Schubert varieties meet transversely, so in this case
$G^r_d(X,(P,a_{\bullet}),(Q,b_{\bullet}))$ is itself a Richardson variety,
and the desired statement follows immediately. 

On the other hand, if $X$ has genus $1$, we elaborate on the argument
given in Lemma 2.1 of \cite{os18}. As mentioned above,
we have nonemptiness in this case if and only if
$a_j + b_{r-j} \leq d$, with equality occuring for at most one $j$.
If equality does occur for some $j_0$, then 
$G^r_d(X,(P,a_{\bullet}),(Q,b_{\bullet}))$ is supported over
the point of $\Pic^d(X)$ corresponding to 
$\sO_{X}(a_{j_0}P+b_{r-j_0} Q)
=\sO_{X}(a_{j_0}P+(d-a_{j_0} Q)$, whereas if the inequality is
strict for all $j$, it maps surjectively onto $\Pic^d(X)$. We refer to
these as the first and second cases, respectively.

Our first task is to verify that in the first case, 
$G^r_d(X,(P,a_{\bullet}),(Q,b_{\bullet}))$ is supported 
scheme-theoretically over the relevant point of $\Pic^d(X)$. Accordingly,
let $T$ be a $k$-scheme, and $(\sL,\cV)$ a $T$-valued point of
$G^r_d(X,(P,a_{\bullet}),(Q,b_{\bullet}))$, so that in particular
$\sL$ is a line bundle on $X \times_k T$ and $\cV$ is a rank-$(r+1)$ 
vector bundle on $T$ with a map to $p_{2*} \sL$. For brevity, write 
$a=a_{j_0}$. Then by definition of imposed ramification, we have that
the rank of the natural maps 
$\cV \to p_{2*} \left(\sL|_{aP \times T}\right)$
and
$\cV \to p_{2*} \left(\sL|_{(d-a)Q \times T}\right)$
are (scheme-theoretically) less than or equal to $j_0$ and $r-j_0$
respectively on $T$. We then see by expanding out minors in local matrix
expressions that the natural map
\begin{equation}\label{eq:restrict}
\cV \to p_{2*} \left(\sL|_{aP \times T}\right) \oplus
p_{2*} \left(\sL|_{(d-a)Q \times T}\right)
= p_{2*} \left(\sL|_{(aP+(d-a)Q) \times T}\right)
\end{equation}
has rank less than or equal to $r$ on $T$.
On the other hand, at any point $t \in T$, the map
$\cV|_t \to \Gamma(X,\sL|_{(aP + (d-a)Q) \times t})$
can have only a $1$-dimensional kernel, so we conclude that
\eqref{eq:restrict} has rank exactly $r$ on $T$, and therefore contains
a rank-$1$ subbundle $\sM$ which gives the kernel universally.
Pulling back to $X \times_k T$, we have maps
$$p_{2*} \sM \to \cV \otimes \sO_{X \times_k T} \to \sL \to 
\sL|_{(aP+(d-a)Q) \times T}$$
such that the composed map to $\sL$ does not vanish identically in any
fiber, but the composed map to $\sL|_{(aP+(d-a)Q) \times T}$
is zero. The latter implies that the map factors through 
a map $p_{2*} \sM \to \sL(-((aP+(d-a)Q)\times T)$ which does not vanish
identically in any fiber. Since both line bundles have relative degree
$0$, we conclude that this map is an isomorphism, and hence that our
$T$-valued point is supported (scheme-theoretically) in the desired
fiber of $\Pic^d(X)$.

Next, in both the first and second cases, if $\sL \in \Pic^d(X)$ is in
the image of $G^r_d(X,(P,a_{\bullet}),(Q,b_{\bullet}))$, then the fiber
over $\sL$ is again
described (scheme-theoretically) as an intersection of two Schubert 
varieties inside a Grassmannian
$G(r+1,\Gamma(X,\sL)) \cong G(r+1,d)$.
We recall from \cite{os18} that when $\sL$ is not of the form 
$\sO(aP+(d-a)Q)$ for some $a$ with $0 \leq a \leq d$, then the intersection 
is again a Richardson variety, of dimension $\rho-1$. If
$\sL=\sO(dP)$, then we still obtain a Richardson variety which will have 
dimension $\rho-1$ if $a_r<d$ and $\rho$ if $a_r=d$ (due to
the difference in this case between the vanishing sequence indexing and
the usual Schubert variety indexing). The same holds if $\sL=\sO(dQ)$,
with $b_r$ in place of $a_r$. The final situation is that 
$\sL=\sO(aP+(d-a)Q)$ for some $a$ with $0 < a < d$. Here, our flags do
not intersect transversely, at precisely the place corresponding to imposing
order-$a$ vanishing at $P$ and order-$(d-a)$ vanishing at $Q$. If either
$a$ does not appear in $a_{\bullet}$ or $d-a$ does not occur in $b_{\bullet}$,
the partial flags used to define the Schubert varieties are still transverse,
so the intersection is still a Richardson variety, of
dimension $\rho-1$. Next, we consider the possibility that $a=a_{j_0}$ and
$d-a=b_{r-j_0}$ for some $j_0$ (i.e., exactly the situation occurring in
our first case). As explained in \cite{os18}, in this case we do not
change the intersection of Schubert varieties if we replace $a_{j_0}$ by
$a-1$, and we then see that we again get a Richardson variety, but this
time of dimension $\rho$. \textit{A priori} this argument was 
set-theoretic, but since replacing $a$ by $a-1$ in $a_{\bullet}$ only
increases the size of the Schubert variety, we have that 
$G^r_d(X,(P,a_{\bullet}),(Q,b_{\bullet}))$ is a closed subscheme of the
Richardson variety, with the same support, and since the latter is reduced
they must agree as schemes. 

This completes the argument for the first case, and for the second case,
the last possibility we have to consider above is that we have 
$a=a_{j_1}$ and $d-a=b_{r-j_2}$ for some $j_1 \neq j_2$. This case was
overlooked in \cite{os18}, and in fact in this case the fibers may
be reducible, but in any case, the dimension of the fiber over $\Pic^d(X)$ 
is at most $\rho-1$ in this situation. The argument can be expressed in
terms of rather standard descriptions of intersections of pairs of
Schubert varieties associated to non-transverse flags (as described e.g.
in \S\S 2.2 and 2.5 and Lemma 2.6 of \cite{va7}), and in this context
turns out to be a special case of the last paragraph of the proof of 
Lemma 6.3 of \cite{os23}. However, because the bound can also be obtained
by a brief direct argument, for the sake of remaining self contained we
include a proof. First, any linear series $(\sL,V)$ in
$G^r_d(X,(P,a_{\bullet}),(Q,b_{\bullet}))$ has some
vanishing sequences $a'_{\bullet} \geq a_{\bullet}$ and $b'_{\bullet} \geq
b_{\bullet}$ at $P$ and $Q$, and it is straightforward to see that $V$ must 
admit a 
basis $s_0,\dots,s_r$ such that $\ord_P s_j=a'_j$ and 
$\ord_Q s_j=b'_{\sigma(r-j)}$ for some permutation $\sigma$ of $0,\dots,r$.
The data of $a'_{\bullet},b'_{\bullet}$ and $\sigma$ is equivalent to 
specifying the dimension of incidence of $V$ with every simultaneous
vanishing condition at $P$ and $Q$, so if we fix $\sL=\sO(aP+(d-a)Q)$, we
obtain a decomposition of the fiber of interest into locally closed 
subsets, each of which we want to show has dimension at most $\rho-1$.
But for each subset, we have a surjection from the space of
$(r+1)$-tuples of sections $s_j$ of $\sL$ satisfying the above conditions
on orders at $P$ and $Q$. This space of tuples has dimension
$$\sum_{j=0}^r \dim \Gamma(X,\sL(-a'_j P -b'_{\sigma(r-j)} Q)),$$
and each term in the sum is equal to $d-a'_j-b'_{\sigma(r-j)}$ except in
the case that $a'_j=a$ and $b'_{\sigma(r-j)}=d-a$, in which case we get
$d-a'_j-b'_{\sigma(r-j)}+1$. Thus, the sum is equal to 
$(r+1)d-\sum_j a'_j - \sum_j b'_j + \epsilon$, where $\epsilon=1$ if there is
some $j$ with $a'_j=a$ and $b'_{\sigma(r-j)}=d-a$, and $\epsilon=0$ otherwise.
The hypotheses of our situation entail that if $\epsilon=1$, then 
$\sigma(r-j)\neq r-j$, so in particular $\sigma \neq \id$. Now, the
fibers from the space of $(r+1)$-tuples to 
$G^r_d(X,(P,a_{\bullet}),(Q,b_{\bullet}))$ always have dimension at 
least $r+1$, corresponding to independent scaling of the $s_j$, and we
see that additional changes of basis are possible (and hence the fiber
dimension is strictly bigger) precisely when $\sigma \neq \id$, since in
this case there is necessarily some $j_1<j_2$ with 
$\sigma(r-j_1)<\sigma(r-j_2)$. We thus
compute that the dimension of each given subset of the fiber is bounded by
$$(r+1)d-\sum_j a'_j - \sum_j b'_j - (r+1)\leq
(r+1)d-\sum_j a_j - \sum_j b_j - (r+1)=\rho-1,$$
as desired.

In the second case, we now have that the map
$G^r_d(X,(P,a_{\bullet}),(Q,b_{\bullet})) \to \Pic^d(X)$ is proper
and surjective, with every fiber having dimension at most $\rho-1$, 
and with the general fiber reduced and irreducible. We also know
\textit{a priori} that every component of 
$G^r_d(X,(P,a_{\bullet}),(Q,b_{\bullet}))$ has dimension at least $\rho$,
so we conclude that every component must have dimension exactly $\rho$,
every fiber must have dimension $\rho-1$, and no component can be supported 
in any fiber. It then follows from
the irreducibility of $\Pic^d(X)$ that 
$G^r_d(X,(P,a_{\bullet}),(Q,b_{\bullet}))$ is irreducible. Now,
in the second case we cannot have $a_r=d$ or $b_r=d$, so we see that
$G^r_d(X,(P,a_{\bullet}),(Q,b_{\bullet}))$ is constructed inside 
the relative Grassmannian $G^r_d(X)$ as an intersection of two
relative Schubert varieties, which are each Cohen-Macaulay.
Since we have shown that this intersection has the expected codimension,
it is Cohen-Macaulay. This implies first that 
$G^r_d(X,(P,a_{\bullet}),(Q,b_{\bullet}))$ is flat over $\Pic^d(X)$.
But then we conclude that $G^r_d(X,(P,a_{\bullet}),(Q,b_{\bullet}))$ 
is generically reduced, since it is irreducible
and the general fiber is reduced. Finally, reducedness follows 
as an additional consequence of Cohen-Macaulayness.
\end{proof}

\section{Limit linear series}

We now use the Eisenbud-Harris theory of limit linear series to prove the 
reducedness portion of Theorem \ref{thm:main}. For the sake of convenience,
we also incorporate the well-known dimension results into our statements.
We begin by specifying the curves we will consider, and recalling the basic
definitions.

\begin{sit}\label{sit:chain} Fix $g,d,n$. Let $Z_1,\dots,Z_n$
be smooth projective curves, with (distinct) $P_i, Q_i$ on $Z_i$ for each
$i$, and let $X_0$ be the nodal curve obtained by gluing $Q_i$ to $P_{i+1}$ 
for $i=1,\dots,n-1$.
\end{sit}

\begin{defn}\label{defn:lls} Given $r,d$ a \textbf{limit linear series}
of rank $r$ and degree $d$ on $X_0$ consists of a tuple $(\sL^i,V^i)$ of 
linear series of rank $r$ and degree $d$ on the $Z_i$, satisfying the
following condition: if $a^i_{\bullet},b^i_{\bullet}$ are the vanishing
sequences of $(\sL^i,V^i)$ at $P_i$ and $Q_i$ respectively, then we require
\begin{equation}\label{eq:eh-ineq} 
b^i_j+a^{i+1}_{r-j} \geq d
\end{equation}
for all $i=1,\dots,n-1$ and $j=0,\dots,r$. If \eqref{eq:eh-ineq} is an
equality for all $i,j$, we say that the limit linear series is
\textbf{refined}.

The space of all such limit linear series on $X_0$ is denoted by
$G^r_d(X_0)$. If we have sequences $a_{\bullet}$ and $b_{\bullet}$,
we also have the closed subscheme 
$G^r_d(X_0,(P_1,a_{\bullet}),(Q_n,b_{\bullet})) \subseteq G^r_d(X_0)$
consisting of limit linear series such that, following the above notation,
we have $a^1_{\bullet} \geq a_{\bullet}$ and $b^n_{\bullet} \geq b_{\bullet}$.
\end{defn}

Any set of vanishing sequences at all of the $P_i,Q_i$ 
satisfying \eqref{eq:eh-ineq} can be decreased to obtain a set of
refined vanishing sequences, so by definition we have that
$G^r_d(X_0,(P_1,a_{\bullet}),(Q_n,b_{\bullet}))$ is the
union over all tuples of complementary vanishing sequences 
$b^i_{\bullet},a^{i+1}_{\bullet}$ of 
$$G^r_d(Z_1,(P_1,a_{\bullet}),(Q_i,b^i_{\bullet})) \times
\left(\prod_{i=2}^{n-1} 
G^r_d(Z_i,(P_i,a^i_{\bullet}),(Q_i,b^i_{\bullet}))\right)
\times G^r_d(Z_n,(P_i,a^i_{\bullet}),(Q_i,b_{\bullet})).$$
We define the scheme structure on
$G^r_d(X_0,(P_1,a^1_{\bullet}),(Q_n,b^n_{\bullet}))$ to be the one
induced by the natural scheme structures on the individual spaces
$G^r_d(Z_i,(P_i,a^i_{\bullet}),(Q_i,b^i_{\bullet}))$. 
To settle questions of reducedness and dimension 
it therefore suffices to consider the individual spaces
$G^r_d(Z_i,(P_i,a^i_{\bullet}),(Q_i,b^i_{\bullet}))$. Note also that by
additivity of the Brill-Noether number (Lemma 3.6 of \cite{e-h1}), 
expected dimension of
the latter spaces implies that the limit linear series space has the 
expected dimension as well.
In particular, when every $Z_i$ has genus at most $1$, 
Proposition \ref{prop:base-cases} has the following immediate corollary.

\begin{cor}\label{cor:lls-reduced-1} Suppose that the genus of each
$Z_i$ is at most $1$, and that for all $i$ with $Z_i$ having genus $1$,
we have that
$P_i-Q_i$ is not $m$-torsion for any $m \leq d$.
Given vanishing sequences
$a_{\bullet}$, $b_{\bullet}$, the space 
$G^r_d(X_0,(P_1,a_{\bullet}),(Q_n,b_{\bullet}))$ of limit linear 
series on $X_0$ is reduced of pure dimension $\rho$.
\end{cor}

We next use the machinery of \cite{os20} and \cite{o-m1} to generalize
to the case of components of higher genus. We first state the following
easy consequence of \cite{o-m1}.

\begin{thm}\label{thm:family-recall}
Let $B$ be the spectrum of a discrete valuation ring, and $\pi:X \to B$ be
a flat, proper family family of curves of genus $g$, with 
$X$ regular, the generic fiber $X_{\eta}$ smooth, and the special fiber
isomorphic to $X_0$. Further assume that $\pi$ has sections $P$, $Q$,
specializing to $P_1$ and $Q_n$ respectively on $X_0$. 

Suppose that for all vanishing
sequences $a_{\bullet}$ and $b_{\bullet}$ with $\widehat{\rho} \geq 0$,
the spaces $G^r_d(X_0, (P_1,a_{\bullet}),(Q_n,b_{\bullet}))$ are
reduced of expected dimension $\rho$, respectively. Then for each
$a_{\bullet},b_{\bullet}$, there is a moduli space
$G^r_d(X/B, (P,a_{\bullet}),(Q,b_{\bullet}))$ proper and flat over $B$,
with special and generic fibers isomorphic to
$G^r_d(X_0, (P_1,a_{\bullet}),(Q_n,b_{\bullet}))$ and
$G^r_d(X_{\eta}, (P_{\eta},a_{\bullet}),(Q_{\eta},b_{\bullet}))$, 
respectively.
\end{thm}

\begin{proof} 
Set $\rho_0=g-(r+1)(g+r-d)$. The hypotheses imply that
$G^r_d(X_0)$ is reduced of dimension $\rho_0$, 
with the refined limit linear series forming a dense open subset, since
any additional ramification results in dimension strictly smaller than
$\rho_0$.
Then according to Corollary 3.4 of \cite{o-m1}, there is a scheme
$G^r_d(X/B)$ which is flat and proper over $B$, with generic fiber 
isomorphic to 
$G^r_d(X_{\eta})$, and with special fiber isomorphic to $G^r_d(X_0)$.
If we now impose vanishing sequence at least $a_{\bullet}$ along $P$
and at least $b_{\bullet}$ along $Q$, we obtain a scheme
$G^r_d(X/B, (P,a_{\bullet}),(Q,b_{\bullet}))$ still proper over $B$,
with generic fiber isomorphic to 
$G^r_d(X_{\eta}, (P_{\eta},a_{\bullet}),(Q_{\eta},b_{\bullet}))$, and
special fiber isomorphic to
$G^r_d(X_0, (P_1,a_{\bullet}),(Q_2,b_{\bullet}))$. By hypothesis, the
latter is reduced, of pure dimension $\rho$. Moreover, the construction
of $G^r_d(X/B)$ inside of an ambient space smooth over $B$ ensures that it has
universal relative dimension at least $\rho_0$ in the sense of
Definition 3.1 of \cite{os21}, and considering the imposition of the 
addition ramification 
conditions inside of the same smooth ambient space we see that
$G^r_d(X/B, (P,a_{\bullet}),(Q,b_{\bullet}))$ has universal relative
dimension at least $\rho$. It then follows from the reducedness of
$G^r_d(X_0, (P_1,a_{\bullet}),(Q_2,b_{\bullet}))$ and Proposition 3.7
of \cite{os21} that 
$G^r_d(X/B, (P,a_{\bullet}),(Q,b_{\bullet}))$ is flat over $B$, as
desired.
\end{proof}

\begin{figure}
\centering
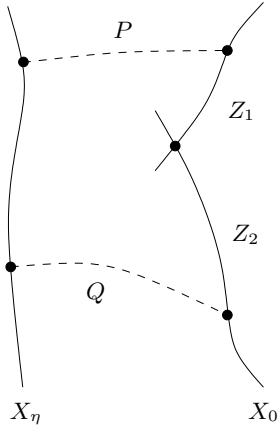
\caption{The degeneration considered in Theorem \ref{thm:family-recall}
in the case $n=2$.}
\label{fig:degen}
\end{figure}

We now conclude the following general reducedness statement. Note that 
the reducedness portion of Theorem 
\ref{thm:main} is the special case with $n=1$.

\begin{cor}\label{cor:lls-reduced-2} Suppose that the components of 
$X_0$ are general as $2$-marked curves. Given vanishing sequences
$a_{\bullet}$, $b_{\bullet}$, the space 
$G^r_d(X_0,(P_1,a_{\bullet}),(Q_n,b_{\bullet}))$ of limit linear 
series on $X_0$ is reduced of pure dimension $\rho$.
\end{cor}

\begin{proof} We induct on genus, assuming that we know
the desired statement for irreducible curves of genus strictly less than $g$,
and proving it for irreducible curves of genus $g$. As discussed above,
this implies the desired statement also for spaces of limit linear series
when all components have genus at most $g$. The base case is that
$g$ is $0$ or $1$, so is handled by Proposition \ref{prop:base-cases}.
Consider $n=2$, with $Z_1$ a smooth
projective curve of genus $1$, and $P_1,Q_1 \in Z_1$ marked points not 
differing by 
$m$-torsion for any $m \leq d$, and with $Z_2$ a general smooth projective
curve of genus $g-1$, and $P_2,Q_2 \in Z_2$ general marked points.
Then let $\pi:X \to B$ be as in Theorem \ref{thm:family-recall}.
By the induction hypothesis, $G^r_d(X_0, (P_1,a_{\bullet}),(Q_2,b_{\bullet}))$
is reduced of pure dimension $\rho$ for all choices of 
$a_{\bullet},b_{\bullet}$, so Theorem \ref{thm:family-recall} gives us
that $G^r_d(X/B, (P,a_{\bullet}),(Q,b_{\bullet}))$ is flat and proper over 
$B$, and then Theorem 12.2.1 of \cite{ega43} implies that the generic fiber
$G^r_d(X_{\eta}, (P_{\eta},a_{\bullet}),(Q_{\eta},b_{\bullet}))$ is
geometrically reduced. Since this condition is open in families,
we conclude that the corollary holds for irreducible curves of genus $g$,
as desired.
\end{proof}

\begin{rem}\label{rem:more-ram} When we consider ramification at three
or more points, the situation is much more delicate, even in genus $0$.
For instance, in positive characteristic, spaces of limit linear series
frequently have greater than the expected dimension due to inseparability
phenomena. Even over the complex numbers, reducedness of spaces of linear 
series with prescribed ramification is quite difficult: the case of genus 
$0$ and $\rho=0$ was proved relatively recently by Mukhin, Tarasov and
Varchenko \cite{m-t-v} while proving the sharper Shapiro-Shapiro conjecture,
but they used a connection to quantum algebra, and there is still no 
direct geometric proof. The case of higher genus then follows via limit
linear series techniques; see \cite{os2}.
\end{rem}

\section{Connectedness}

Having established the necessary reducedness results, we can now prove the 
connectedness portion of Theorem \ref{thm:main}.

We have the following preliminary lemma.

\begin{lem}\label{lem:meets-open} Suppose $\pi:Y \to B$ is flat and 
proper, with geometrically reduced fibers, and $B$ is locally Noetherian
and connected. Let $Z \subseteq Y$ be a closed subscheme which is also
flat over $B$. If there exists $b \in B$ such that every connected
component of $Y_b$ meets $Z$, then the same is true for all $b \in B$.
\end{lem}

\begin{proof} As $B$ is assumed connected, it is enough to show that 
the set of $b \in B$ such that every connected component of $Y_b$
meets $Z$ is closed under both specialization and generization.
Observe that the condition on $Y_b$ is invariant under extension of
base field, because the image of a connected component under a base
field extension is a connected component (see for instance Tag 04PZ
of \cite{stacks-proj}, although the proof is easier in the finite type
case).
Thus, the statement reduces to the case that $B$ is the
spectrum of a DVR. In fact, following the reasoning in the proof of
Proposition 15.5.7 of \cite{ega43}, we may further assume that the
connected components of the fibers are geometrically connected, and
we then obtain that $Y$ decomposes as a disjoint union of connected 
components, each with a unique connected component in both fibers.
In this situation, the desired statement follows immediately from the
flatness and properness of $Z$.
\end{proof}

We also have the following lemma on subadditivity of $\widehat{\rho}$, which
follows easily from the definitions.

\begin{lem}\label{lem:subadditive} Given $g,r,d$, $a_{\bullet}$, 
$b_{\bullet}$ as in Theorem \ref{thm:bg}, suppose we have also
$b^1_{\bullet}$ and $a^2_{\bullet}$ complementary vanishing sequences,
and $g_1,g_2$ nonnegative with $g_1+g_2=g$. Define 
$\widehat{\rho}_1$ using $g_1$ and $b^1_{\bullet}$ in place of $g$ and 
$b_{\bullet}$, and define 
$\widehat{\rho}_2$ using $g_2$ and $a^2_{\bullet}$ in place of $g$ and 
$a_{\bullet}$. Then we have
$$\widehat{\rho} \geq \widehat{\rho}_1+\widehat{\rho}_2,$$
with equality occurring if and only if for every $j$ with
$a_j+b^1_{r-j}>d-g_1$, we have $a^2_j+b_{r-j} \geq d-g_2$, and
for every $j$ with
$a^2_j+b_{r-j}>d-g_2$, we have $a_j+b^1_{r-j} \geq d-g_1$.
\end{lem}

In order to prove the connectedness portion of Theorem \ref{thm:main},
we will need to induct on a more refined statement that involves also
the case $\widehat{\rho}=0$. Specifically, we prove the following:

\begin{thm}\label{thm:connected} In the situation of Theorem \ref{thm:bg},
if we have $\widehat{\rho} \geq 1$, then the moduli space 
$G^r_d(X,(P,a_{\bullet}),(Q,b_{\bullet}))$ is connected. 

If $\widehat{\rho}=0$, and if for some $j_0$ we can define a vanishing sequence
$\bar{a}_{\bullet}$ by replacing $a_{j_0}$ with $a_{j_0}+1$, while 
maintaining the condition $\widehat{\rho}=0$, then every connected component of 
$G^r_d(X,(P,a_{\bullet}),(Q,b_{\bullet}))$ has nonempty intersection with
the closed subscheme
$G^r_d(X,(P,\bar{a}_{\bullet}),(Q,b_{\bullet}))$.
\end{thm}

Note that the hypotheses on $\bar{a}_{\bullet}$ in the second part of the
theorem imply in particular that $\rho \geq 1$.

\begin{proof} We begin by proving the second statement by induction on $g$.
The base case that $g \leq 1$ is satisfied trivially because in this case
$G^r_d(X,(P,a_{\bullet}),(Q,b_{\bullet}))$ is always connected.
For the induction step, we again consider a family $\pi:X \to B$ just as 
in the proof of Corollary \ref{cor:lls-reduced-2}. We now know that
$G^r_d(X_0, (P_1,a_{\bullet}),(Q_2,b_{\bullet}))$ is reduced and 
that 
$G^r_d(X/B, (P,a_{\bullet}),(Q,b_{\bullet}))$ is flat and proper over $B$, 
so according to Proposition 15.5.7 of \cite{ega43}, the number of 
(geometrically) connected components of 
$G^r_d(X/B, (P,a_{\bullet}),(Q,b_{\bullet}))$ is the same in the special
and generic fibers. The same statements also hold for
$G^r_d(X/B, (P,\bar{a}_{\bullet}),(Q,b_{\bullet}))$.
Our main task is to
show that every connected component of the limit linear series space
$G^r_d(X_0,(P_1,a_{\bullet}),(Q_2,b_{\bullet}))$ has nonempty intersection 
with the closed subscheme
$G^r_d(X_0,(P_1,\bar{a}_{\bullet}),(Q_2,b_{\bullet}))$. We will analyze the
situation on each part
$$G^r_d(X_0,(P_1,a_{\bullet}),(Q_2,b_{\bullet}))^{b^1_{\bullet},a^2_{\bullet}}
:=G^r_d(Z_1,(P_1,a_{\bullet}),(Q_1,b^1_{\bullet})) \times
G^r_d(Z_2,(P_2,a^2_{\bullet}),(Q_2,b_{\bullet}))$$ 
of $G^r_d(X_0,(P_1,a_{\bullet}),(Q_2,b_{\bullet}))$, as $b^1_{\bullet}$,
$a^2_{\bullet}$ range over complementary vanishing sequences with
$\widehat{\rho}_1$ and $\widehat{\rho}_2$ both nonnegative.

First, if $a_{j_0}+b^1_{r-j_0}<d-1$, we see that the closed subset
$$G^r_d(Z_1,(P_1,\bar{a}_{\bullet}),(Q_1,b^1_{\bullet})) \times
G^r_d(Z_2,(P_2,a^2_{\bullet}),(Q_2,b_{\bullet}))$$ 
will still be nonempty,
and because $G^r_d(Z_1,(P_1,a_{\bullet}),(Q_1,b^1_{\bullet}))$ is 
connected, we immediately obtain the desired statement, so it is enough
to consider the cases that $a_{j_0}+b^1_{r-j_0}=d$ or 
$a_{j_0}+b^1_{r-j_0}=d-1$. On the other hand, 
our hypothesis that $\widehat{\rho} \geq 0$ is maintained under replacing
$a_{\bullet}$ with $\bar{a}_{\bullet}$ implies that $a_{j_0}+b_{r-j_0}<d-g$,
and then by Lemma \ref{lem:subadditive} we see that we cannot have
$a_{j_0}+b^1_{r-j_0}=d$. Thus, it remains to consider the case that
$a_{j_0}+b^1_{r-j_0}=d-1$. In this case, we necessarily have
$a^2_{j_0}+b_{r-j_0}<d-(g-1)$, and we claim that if $j_0<r$, we must have
$a^2_{j_0+1}>a^2_{j_0}+1$: indeed, we must have $a_{j_0+1}>a_{j_0}+1$
by the validity of $\bar{a}_{\bullet}$, so the only way we could have
$a^2_{j_0+1}=a^2_{j_0}+1$
would be if $a_{j_0+1}+b^1_{r-j_0-1}=d$. But then because $\widehat{\rho}=0$,
Lemma \ref{lem:subadditive} would imply that 
$a^2_{j_0+1}+b_{r-j_0-1}\geq d-(g-1)$, yielding the contradiction that
$b_{r-j_0} \leq b_{r-j_0-1}$.
Thus, we can obtain $\bar{b}^1_{\bullet}$ and $\bar{a}^2_{\bullet}$ by
replacing $b^1_{r-j_0}$ with $b^1_{r-j_0}-1$ and $a^2_{j_0}$ with
$a^2_{j_0}+1$ respectively, and this maintains the condition 
$\widehat{\rho}_2 \geq 0$. The induction hypothesis implies that
every connected component of 
$G^r_d(Z_2,(P_2,a^2_{\bullet}),(Q_2,b_{\bullet}))$ meets
$G^r_d(Z_2,(P_2,\bar{a}^2_{\bullet}),(Q_2,b_{\bullet}))$, so we conclude 
that every connected component of the part
$G^r_d(X_0,(P_1,a_{\bullet}),(Q_2,b_{\bullet}))^{b^1_{\bullet},a^2_{\bullet}}$
meets the part
$G^r_d(X_0,(P_1,a_{\bullet}),(Q_2,b_{\bullet}))^{\bar{b}^1_{\bullet},
\bar{a}^2_{\bullet}}$.
But now we have reduced to the previously handled case, so every
connected component of the part
$G^r_d(X_0,(P_1,a_{\bullet}),(Q_2,b_{\bullet}))^{\bar{b}^1_{\bullet},
\bar{a}^2_{\bullet}}$
meets the subset
$G^r_d(Z_1,(P_1,\bar{a}_{\bullet}),(Q_1,\bar{b}^1_{\bullet})) \times
G^r_d(Z_2,(P_2,\bar{a}^2_{\bullet}),(Q_2,b_{\bullet}))$, proving the
desired statement for $X_0$.
The desired statement for 
$G^r_d(X_{\eta}, (P_{\eta},a_{\bullet}),(Q_{\eta},b_{\bullet}))$ then
follows from Lemma \ref{lem:meets-open}, and considering a versal family
of curves containing $X_{\eta}$ and consisting entirely of curves $X$
and $P,Q$ such
that $G^r_d(X, (P,a_{\bullet}),(Q,b_{\bullet}))$ is geometrically reduced
of dimension $\rho$, the same lemma implies the desired statement for a
general $2$-marked curve.

We next show that when $X$ is a general $2$-marked curve, and 
$\widehat{\rho} \geq 1$, the space 
$G^r_d(X,(P,a_{\bullet}),(Q,b_{\bullet}))$ is connected. The proof will
again be by induction on genus, with the base cases $g=0,1$ having been
proved in Proposition \ref{prop:base-cases}. For the induction step,
we consider the same family $\pi:X \to B$ as before. Using the aforementioned
constancy of geometric connected components (again, both in the family 
$\pi$, and in a suitable versal family), we see that it is enough to
show that $G^r_d(X_0, (P_1,a_{\bullet}),(Q_2,b_{\bullet}))$ is connected.

Now, for any vanishing sequences $b^1_{\bullet}$ and $a^2_{\bullet}$,
we have by Proposition \ref{prop:base-cases} that
$G^r_d(Z_1, (P_1,a_{\bullet}),(Q_1,b^1_{\bullet}))$ is connected when
it is nonempty, and by the induction hypothesis,
$G^r_d(Z_2, (P_2,a^2_{\bullet}),(Q_2,b_{\bullet}))$ is connected when
$\widehat{\rho}_2>0$. Moreover, by the statement we have just proved, even
when $\widehat{\rho}_2=0$, we have that if we increase one term of 
$a^2_{\bullet}$ while maintaining the condition that $\widehat{\rho}_2=0$,
every connected component 
$G^r_d(Z_2, (P_2,a^2_{\bullet}),(Q_2,b_{\bullet}))$ will meet the resulting
closed subscheme. We will consider complementary pairs of
$b^1_{\bullet}$ and $a^2_{\bullet}$ with $\widehat{\rho}_1 \geq 0$ and
$\widehat{\rho}_2 \geq 0$ and show that the parts
$G^r_d(X_0,(P_1,a_{\bullet}),(Q_2,b_{\bullet}))^{b^1_{\bullet},a^2_{\bullet}}$
corresponding to any two such pairs can
be connected to one another through a sequence of moves which increase
exactly one entry of $b^1_{\bullet}$ or $a^2_{\bullet}$ at a time.

There are two types of moves we consider. For the first, we assume that
$\rho_1 > 0$. Let $j_0$
be minimal such that $a_{j_0}+b^1_{r-j_0}$ attains its minimum value,
which is necessarily at most $d-1$. By the minimality assumptions, either
$j_0=0$ or $b^1_{r-j_0+1} \geq b^1_{r-j_0}+2$; in either case, we can
replace $b^1_{r-j_0}$ with $b^1_{r-j_0}+1$ to obtain a valid sequence which
still has $\widehat{\rho}_1 \geq 0$, but with $\rho_1$ decreased by $1$. 
We can then replace
$a^2_{j_0}$ with $a^2_{j_0}-1$, and we obtain a new pair of complementary
sequences which still satisfy the nonemptiness condition. 

For the second type of move, suppose that $\widehat{\rho}_2=0$, and that if
also $\widehat{\rho}_1=0$, then the $j$ with $a_j+b^1_{r-j}=d$ satisfies
$a^2_j+b_{r-j} \geq d-(g-1)$. Then the second type of move is, for
a certain $j_0$, to replace $a^2_{j_0}$ by $a^2_{j_0}-1$, and 
$b^1_{r-j_0}$ by $b^1_{r-j_0}+1$. This will result in a case with
$\widehat{\rho}_1 \geq 0$ still, but $\widehat{\rho}_2>0$. The desired $j_0$
must therefore satisfy the following conditions:
$a_{j_0}+b^1_{r-j_0} \leq d-1$; $a^2_{j_0}+b_{r-j_0}>d-(g-1)$; and 
either $j_0=0$ or $a^2_{j_0-1}<a^2_{j_0}-1$. Moreover, if $\widehat{\rho}_1=0$,
we need $a_{j_0}+b^1_{r-j_0}<d-1$. Now, if $\widehat{\rho}_1>0$, the
condition $a_{j_0}+b^1_{j_0} \leq d-1$ is automatic, and we obtain such a 
$j_0$ by letting $j_0$ be minimal with $a^2_{j_0}+b_{r-{j_0}}$ maximal.
If $\widehat{\rho}_1=0$, because $\widehat{\rho}>0$ Lemma \ref{lem:subadditive}
implies that we must have some $j$
with $a^2_{j}+b_{r-j}>d-(g-1)$, and $a_{j}+b^1_{r-j}<d-1$.  Let $j_0$
be minimal such that $a^2_{j_0}+b_{r-j_0}$ attains the maximal value
subject to the condition that $a_{j_0}+b^1_{r-j_0}<d-1$. By construction,
if $j_0>0$, we either have $a^2_{j_0-1}+b_{r-j_0+1}<a^2_{j_0}+b_{r-j_0}$ or
$a_{j_0-1}+b^1_{r-j_0+1} > a_{j_0}+b^1_{r-j_0}$ and in either case we see
that $a^2_{j_0-1}<a^2_{j_0}-1$, as claimed. Thus, we have produced $j_0$
satisfying the desired conditions.

Now, suppose we start with $\widehat{\rho}_1>0$; applying the first type of
move repeatedly necessarily brings us eventually to the single case that 
$a_j+b^1_{r-j}=d-1$ for all $j$. Then applying the second type of move once,
we see that this case is connected to a case with $\widehat{\rho}_2>0$.
Finally, we show that any case with $\widehat{\rho}_1=0$ is connected to a
case with $\widehat{\rho}_1>0$. Let $j_0$ be the value with
$a_{j_0}+b^1_{r-j_0}=d$; then we must have either $j_0=0$ or
$a_{j_0-1}+b^1_{r-j_0+1}=d-1$, so in either case we can decrease 
$b^1_{r-j_0}$ by $1$. If $a^2_{j_0}+b_{r-j_0}<d-(g-1)$, 
we can increase $a^2_{j_0}$ by $1$, and will still necessarily have 
$\widehat{\rho}_2 \geq 0$, connecting us to the desired case with
$\widehat{\rho}_1>0$. On the other hand, if $a^2_{j_0}+b_{r-j_0} \geq d-(g-1)$, 
then either $\widehat{\rho}_2 >0$, or we can apply the second type of move
to connect to a case with $\widehat{\rho}_2>0$. Once $\widehat{\rho}_2>0$,
we again find that we can increase $a^2_{j_0}$ by $1$, yielding the
desired case with $\widehat{\rho}_1>0$. To summarize, we have seen first
that all cases are connected to one another, and second, that at least
some cases have $\widehat{\rho}_2>0$, so that the corresponding product
space is connected. Thus, we have represented 
$G^r_d(X_0, (P_1,a_{\bullet}),(Q_2,b_{\bullet}))$ as a union of
the parts 
$G^r_d(X_0,(P_1,a_{\bullet}),(Q_2,b_{\bullet}))^{b^1_{\bullet},a^2_{\bullet}}$,
where each part may be represented as the vertex of a connected graph, 
and the edges correspond to the above-described moves. 
Using what we have already proved for the case $\widehat{\rho}_2=0$, 
we see that the part corresponding to each vertex has every connected
component meeting the part corresponding to any adjacent vertex,
and we also have that at least one of the parts is connected. It follows 
that the union is connected, as desired.
\end{proof}

We now easily conclude our main theorem.

\begin{proof}[Proof of Theorem \ref{thm:main}] We have already proved
the first statement, as the $n=1$ case of Corollary \ref{cor:lls-reduced-2}.
Theorem \ref{thm:connected} gives us that
$G^r_d(X,(P,a_{\bullet}),(Q,b_{\bullet}))$ is connected when $X$ is a 
general $2$-marked curve and $\widehat{\rho} \geq 1$.
Finally, because the moduli spaces of linear series are univerally open over 
the locus on $\cM_g$ on which they have the expected dimension (as 
explained for instance in the proof of Theorem \ref{thm:family-recall}), we 
can conclude the connectedness assertion of Theorem \ref{thm:main} by
specialization from general curves, using Corollary 15.5.4 of \cite{ega43}.
\end{proof}

\begin{ex}\label{ex:disconnected} We see that the connectedness statement
fails quite quickly when we have $\widehat{\rho}=0$. Indeed, 
consider the case $g=2$, $r=2$, $d=6$, with $a_{\bullet}=0,2,3$ and
$b_{\bullet}=0,3,5$. This has $\rho=1$ but $\widehat{\rho}=0$. Consider
a degeneration to two genus-$1$ components. We see that the only 
possibilities for $b^1_{\bullet}$ which have both
$\widehat{\rho}_1 \geq 0$ and $\widehat{\rho}_2 \geq 0$ are
$1,3,6$; $2,3,6$; $1,4,5$; or $2,4,5$. The first pair and second pair
correspond to intersecting components of the limit linear series space,
but neither component from the first pair intersects either component
of the second pair, so we see that the limit linear series space is
disconnected in this case. Since the space is nonetheless reduced of
the expected dimension, it follows that 
$G^r_d(X,(P,a_{\bullet}),(Q,b_{\bullet}))$ likewise has two connected
components for $(X,P,Q)$ a general $2$-marked curve of genus $2$.

We also consider the second part of Theorem \ref{thm:connected} in
this case. The only possible choice of $\bar{a}_{\bullet}$ will be $0,2,4$.
This results in two limit linear series, with
$b^1_{\bullet}$ given by $1,3,6$ or $1,4,5$. Thus, we see that only two
of the four irreducible components of 
$G^r_d(X_0,(P_1,a_{\bullet}),(Q_2,b_{\bullet}))$ meet
$G^r_d(X_0,(P_1,\bar{a}_{\bullet}),(Q_2,b_{\bullet}))$, but that, as
the theorem asserts, both connected components meet it. 
\end{ex}

Finally, we note that we can conclude connectedness more generally for
spaces of limit linear series:

\begin{cor}\label{cor:lls-connected} In Situation \ref{sit:chain}, 
given also vanishing sequences $a_{\bullet}$, $b_{\bullet}$, suppose
also that $\widehat{\rho} \geq 1$, and that the space 
$G^r_d(X_0,(P_1,a_{\bullet}),(Q_n,b_{\bullet}))$ is reduced of dimension
$\rho$. Then
$G^r_d(X_0,(P_1,a_{\bullet}),(Q_n,b_{\bullet}))$ is connected.
\end{cor}

Note that by Corollary \ref{cor:lls-reduced-2}, the dimension and 
reducedness hypotheses are satisfied in particular when the components of 
$X_0$ are general as $2$-marked curves. 

\begin{proof} Knowing connectedness for a general smooth curve, the desired
statement follows by specialization just as in the proof of the last part 
of Theorem \ref{thm:main}.
\end{proof} 

\begin{rem}\label{rem:conn-combinatorial} Even when $\widehat{\rho}=0$, our 
techniques reduce the problem of finding the number of connected components 
of $G^r_d(X,(P,a_{\bullet}),(Q,b_{\bullet}))$ for a general $2$-marked curve
to a purely combinatorial question. Indeed, if we use
Theorem \ref{thm:family-recall}, Corollary \ref{cor:lls-reduced-1}, and 
Proposition 15.5.7 of \cite{ega43}, we see that the answer is determined
by the number of connected components of 
$G^r_d(X_0,(P_1,a_{\bullet}),(Q_n,b_{\bullet}))$, where $X_0$ is a chain
of curves of genus $1$ with suitably general marked points. As is well
known (and follows for instance from Proposition \ref{prop:base-cases}),
the irreducible components of 
$G^r_d(X_0,(P_1,a_{\bullet}),(Q_n,b_{\bullet}))$ correspond to choices of
tuples of complementary vanishing sequences at the nodes which have
$\widehat{\rho}_i \geq 0$ for all $i$, and we can likewise determine which
pairs of components have nonempty intersection by looking at the vanishing
sequences involved, so the question becomes entirely combinatorial.
\end{rem}

\bibliographystyle{amsalpha}
\bibliography{gen}

\end{document}

%% file: headers.tex
\usepackage{amssymb,euscript,amsmath, mathrsfs}
\usepackage{bm}
\usepackage{rotating}
\usepackage[table]{xcolor}

\makeindex

\newcounter{ENUM}

\newcommand{\margh}[1]{}

\input xy
\xyoption{all}
\CompileMatrices

\def\cM{{\mathcal M}}

\def\cV{{\mathcal V}}

\def\sL{{\mathscr L}}
\def\sM{{\mathscr M}}
\def\sO{{\mathscr O}}

\def\Pic{\operatorname{Pic}}

\def\id{\operatorname{id}}

\def\ord{\operatorname{ord}}

\newtheorem{thm}{Theorem}[section]
\newtheorem{prop}[thm]{Proposition}
\newtheorem{lem}[thm]{Lemma}
\newtheorem{cor}[thm]{Corollary}

\theoremstyle{definition}
\newtheorem{defn}[thm]{Definition}
\newtheorem{ques}[thm]{Question}

\newtheorem{ex}[thm]{Example}
\newtheorem{sit}[thm]{Situation}
\theoremstyle{remark}

\newtheorem{rem}[thm]{Remark}

\numberwithin{equation}{section}

%% file: degen.pdf_t
\begin{picture}(0,0)%
\includegraphics{degen.pdf}%
\end{picture}%
\setlength{\unitlength}{3315sp}%
\begingroup\makeatletter\ifx\SetFigFont\undefined%
\gdef\SetFigFont#1#2#3#4#5{%
  \reset@font\fontsize{#1}{#2pt}%
  \fontfamily{#3}\fontseries{#4}\fontshape{#5}%
  \selectfont}%
\fi\endgroup%
\begin{picture}(1946,3192)(3557,-4591)
\put(5223,-2274){\makebox(0,0)[lb]{\smash{{\SetFigFont{9}{10.8}{\rmdefault}{\mddefault}{\updefault}{\color[rgb]{0,0,0}$Z_1$}%
}}}}
\put(5254,-3201){\makebox(0,0)[lb]{\smash{{\SetFigFont{9}{10.8}{\rmdefault}{\mddefault}{\updefault}{\color[rgb]{0,0,0}$Z_2$}%
}}}}
\put(3594,-4526){\makebox(0,0)[lb]{\smash{{\SetFigFont{9}{10.8}{\rmdefault}{\mddefault}{\updefault}{\color[rgb]{0,0,0}$X_{\eta}$}%
}}}}
\put(5376,-4527){\makebox(0,0)[lb]{\smash{{\SetFigFont{9}{10.8}{\rmdefault}{\mddefault}{\updefault}{\color[rgb]{0,0,0}$X_0$}%
}}}}
\put(4368,-1678){\makebox(0,0)[lb]{\smash{{\SetFigFont{9}{10.8}{\rmdefault}{\mddefault}{\updefault}{\color[rgb]{0,0,0}$P$}%
}}}}
\put(4166,-3635){\makebox(0,0)[lb]{\smash{{\SetFigFont{9}{10.8}{\rmdefault}{\mddefault}{\updefault}{\color[rgb]{0,0,0}$Q$}%
}}}}
\end{picture}%